\documentclass{jvarticle}

\usepackage{amsmath,amsfonts}
\usepackage{graphicx}
\usepackage{hyperref}
\usepackage[round]{natbib}

\newcommand{\C}{\mathbb{C}}
\newcommand{\CO}{\mathcal{O}}
\newcommand{\CU}{\mathcal{U}}
\let\d=\partial

\def\e{\mathrm{e}}
\let\eps=\varepsilon
\DeclareMathOperator{\diag}{diag}

\def\lmin{\lambda_{\rm min}}
\def\MVM{\mathrm{MVM}}

\let\phi=\varphi
\def\quark{\setbox0\hbox{$x$}\hbox to\wd0{\hss$\cdot$\hss}}
\newcommand{\R}{\mathbb{R}}
\newcommand{\Rp}{\R^p}
\newcommand{\Rpp}{\R^{p\times p}}

\newcommand{\T}{\mathbb{T}}

\def\VM{\mathrm{VM}}

\hypersetup{pdftitle={Some Fundamental Properties of a Multivariate von Mises Distribution},
  pdfauthor={Mardia, Kanti V. and Voss, Jochen}}

\begin{document}

\title{Some Fundamental Properties of a \\ Multivariate von Mises Distribution}
\author{Kanti V. Mardia and Jochen Voss}
\date{17th February 2012}
\maketitle

\footnotetext[1]{Department of Statistics, University of Leeds, Leeds,
LS2 9JT, United Kingdom}
\footnotetext[2]{corresponding author}

\begin{abstract}
  In application areas like bioinformatics multivariate distributions
  on angles are encountered which show significant clustering.  One
  approach to statistical modelling of such situations is to use
  mixtures of unimodal distributions.  In the literature
  \citep{MaKeZhaTayHa}, the multivariate von Mises distribution, also
  known as the multivariate sine distribution, has been suggested for
  components of such models, but work in the area has been hampered by
  the fact that no good criteria for the von Mises distribution to be
  unimodal were available.  In this article we study the question
  about when a multivariate von Mises distribution is unimodal.  We
  give sufficient criteria for this to be the case and show examples
  of distributions with multiple modes when these criteria are
  violated.  In addition, we propose a method to generate samples from
  the von Mises distribution in the case of high concentration.
\end{abstract}


\section{Introduction}

In biochemistry it is well known that the structure of macro-molecules
such as proteins, DNA, and RNA can be described in terms of
conformational angles.  For proteins, these angles could be the
dihedral and bond angles describing the conformation of the backbone
together with additional angles for the configuration of the
side chains \citep[see \textit{e.g.}][]{BraToo98}.  Data sets consist of
the angles to describe each monomer in a macro-molecule, the number of
angles required to give the conformation of a monomer determines the
dimensionality of the problem.  In non-coding RNA there can be 7 or~8
dihedral angles of importance per amino acid \citep{FrelsenEtAl09}
and, if the side chains angles are included, many angles are required
for amino acids in proteins \citep[\textit{e.g.}][]{HarderEtAl10}.
The resulting distributions on angles are multivariate, often highly
structured, featuring various modes together with regions excluded by
steric constraints \citep[\textit{e.g.}][]{MaKeZhaTayHa}.

One way to approach the statistical modelling of such multimodal,
multivariate distributions is to use mixture models
with unimodal components.  In the Euclidean
space $\R^p$, an obvious choice for the components is
to use normal distributions with appropriately chosen covariance
matrices.  For angular data, as considered in this article, the choice
of component distribution in less clear, but a simple analogue of the
multivariate normal distribution is the multivariate von~Mises
distribution \citep{Mardia08a}.  This distribution is
suggested for mixture modelling in \citet{MaKeZhaTayHa}.  In order for
a mixture model to be a useful description of a multimodal
distribution, it is essential that the component distribution is
unimodal.  In case of the multivariate von Mises distribution,
this constraint excludes some of the parameter range.
Previous work has been
complicated by the problem that no characterisation of
the parameter values corresponding to the unimodal case was available.
To solve this
problem, this article provides sufficient criteria for the
multivariate von
Mises distribution to be unimodal and we show examples of
distributions with multiple modes (where these criteria are violated).
It should be noted that univariate circular distributions are well established
\citep[see, for example,][]{MaJu00} but understanding of multicircular distributions
is still evolving.

The multivariate von~Mises distribution, first introduced
in~\citet{Mardia08a} and also known as the multivariate sine distribution,
is denoted by~$\MVM(\mu,\kappa,\Lambda)$.  It is
a distribution on the torus $\T^p = [0, 2\pi)^p$ and is given by the
density (w.r.t.\ the uniform distribution on angles)
\begin{equation}\label{eq:MVM}
  \phi(\theta;\mu,\kappa,\Lambda)
  = \frac{1}{Z(\kappa,\Lambda)}
    \exp\bigl(
      \kappa^\top c(\theta)
      + \frac12 s(\theta)^\top\Lambda s(\theta)
    \bigr)
\end{equation}
for all $\theta\in\T^p$.  Here $Z(\kappa,\Lambda)$ is the normalisation
constant and we use the abbreviations
\begin{equation*}
c_i(\theta) = \cos(\theta_i-\mu_i), \quad
s_i(\theta) = \sin(\theta_i-\mu_i)
\end{equation*}
for $i=1, \ldots, p$.  The parameters of the distribution are the ``mean''
$\mu\in\T^p$, the ``concentration parameter'' $\kappa\in\Rp$ with
$\kappa_i\geq 0$ for $i=1, \ldots, p$ and
$\Lambda=(\lambda_{ij})\in\Rpp$ with $\Lambda^\top=\Lambda$ and
$\lambda_{ii}=0$ for $i=1, \ldots, p$.

From the form of the density it is obvious that whenever $\kappa$ is
``large'' compared to $\Lambda$, the density will have exactly one
maximum (where the vector $c(\theta)$ is approximately aligned with
$\kappa$) and exactly one minimum (where $c(\theta)$ is approximately
aligned with $-\kappa$).  This effect is studied in
section~\ref{S:high} where we give a sufficient criterion for the
distribution to be unimodal.  Conversely, for small $\kappa$ the
quadratic term $s^\top(\theta) \Lambda s(\theta)$ in the
density~$\phi$ dominates and one expects the occurrence of multimodal
distributions.  This situation is studied in section~\ref{S:low} where
we show, by example, that a high number of modes is possible
even in low dimensions.  Finally, in section~\ref{S:sampling}, we give
an algorithm for generating samples of a~$\MVM(\mu,\kappa,\Lambda)$
distribution for the unimodal case.  This will be required as part of
any algorithm to sample from a mixture model with $\MVM(\mu, \kappa, \Lambda)$
components.

\section{High Concentration}
\label{S:high}

In this section we derive a sufficient criterion for the
$\MVM(\mu,\kappa,\Lambda)$ to be unimodal.  Since the exponential
function $\exp$ in the density~\eqref{eq:MVM}
is strictly monotonically increasing and since the
normalisation constant $Z(\kappa,\Lambda)$ does not depend on
$\theta$, it suffices to consider the extrema of
\begin{equation}\label{eq:f}
  f(\theta) = \kappa^\top c(\theta) + \frac12 s(\theta)^\top\Lambda s(\theta)
\end{equation}
instead.  These can be found by setting the partial derivatives
\begin{equation}\label{eq:grad}
  \d_if(\theta)
  = -\kappa_i s_i(\theta)
    + c_i(\theta) \sum_{k=1}^p \lambda_{ik} s_k(\theta)
\end{equation}
equal to~$0$: Since $\T^p$ is a compact, closed manifold, all local
extrema of $f$ are located at $\theta\in \T^p$ with $\d_if(\theta) =
0$ for $i = 1, \ldots, p$, \textit{i.e.}\ at critical points of $f$.

To characterise the critical points of $f$, we consider the second
derivatives
\begin{equation}\label{eq:Hess}
  \d_{ij}f(\theta)
  = -\Bigl(\kappa_i c_i(\theta)
           + s_i(\theta) \sum_{k=1}^p \lambda_{ik} s_k(\theta) \Bigr)\delta_{ij}
    + c_i(\theta) \lambda_{ij} c_j(\theta)
\end{equation}
where $\delta_{ij}$ denotes the Kronecker delta.  If the Hessian
matrix $H_f(\theta) = (\d_{ij}f(\theta))_{i,j}$ at a critical point
$\theta$ is negative definite, $\theta$ is a local maximum of $f$ and
thus of $\phi(\quark;\mu,\kappa,\Lambda)$; if $H_f(\theta)$ is
positive definite, $\theta$ is a local minimum; finally, if
$H_f(\theta)$ has both positive and negative eigenvalues, the point
$\theta$ is a saddle point.

For reference in the arguments below, we note that the biggest
eigenvalue $\lambda_{\mathrm{max}}$ of a symmetric matrix
$A=(a_{ij})\in\Rpp$ satisfies
\begin{equation*}
  \lambda_{\mathrm{max}}
  = \sup_{x\in\Rp, |x|=1} x^\top Ax
  \geq \max_{i=1, \ldots, p} e_i^\top Ae_i
  = \max_{i=1, \ldots, p} a_{ii}
\end{equation*}
where $(e_1, \ldots, e_p)$ denotes the standard basis in~$\R^p$.
In particular, if the Hessian matrix at a critical point $\theta$ has
a positive diagonal element, it has at least one positive eigenvalue
and thus $\theta$ cannot be a local maximum.  Similarly,
the smallest eigenvalue $\lmin$ satisfies $\lmin \leq \min_{i=1,
  \ldots, p} a_{ii}$ and if $H_f(\theta)$ has a negative diagonal
element, $\theta$ cannot be a local minimum.

\begin{proposition}\label{P:0 is max}
  Assume that the matrix
  \begin{equation*}
    P = \diag(\kappa_1, \ldots, \kappa_p) - \Lambda
  \end{equation*}
  is positive definite.  Then the global maximum of $\phi =
  \phi(\quark;\mu,\kappa,\Lambda)$ is attained at $\theta = \mu$ and
  $\phi$ has no other (local) maxima.
\end{proposition}

\begin{proof}
  For $\theta = \mu$ we get $\nabla f(\mu) = 0$ and $H_f(\mu) = - P$;
  by assumption, this matrix is negative definite and thus $\theta =
  \mu$ is a \emph{local} maximum.  We now show that this is the only
  local, and thus the \emph{global}, maximum of $f$.

  Since $P$ is positive, the smallest eigenvalue $\lmin$ of $P$
  satisfies $0 < \lmin \leq \min_{i=1, \ldots, p} P_{ii} = \min_{i=1,
    \ldots, p}\kappa_i$ and thus we have $\kappa_i > 0$ for $i=1,
  \ldots, p$.  From equation~\eqref{eq:grad} we see that $\d_i
  f(\theta) = 0$ implies $c_i \neq 0$ and consequently $\sum
  \lambda_{ik} s_k = \kappa_i s_i / c_i$.  Substituting this into the
  expression for $\d_{ij}f$ in~\eqref{eq:Hess} we find that the
  Hessian matrix $H_f$ at a critical point has the elements
  \begin{equation*}
    \d_{ij}f(\theta)
    = - \kappa_i \bigl( x_i + \frac{s_i^2}{c_i}\bigr) \delta_{ij} + c_i\lambda_{ij}c_j
    = -\frac{\kappa_i}{c_i}\delta_{ij} + c_i\lambda_{ij}c_j
  \end{equation*}
  where we write $c$ for $c(\theta)$ and $s$ for $s(\theta)$ to
  improve readability.
  If a critical point $\theta$ has $c_i(\theta) < 0$ for an index
  $i\in\{1, \ldots, p\}$, then $H_f(\theta)_{ii} = -\kappa_i/c_i > 0$
  and thus $\theta$ cannot be a local maximum.  Therefore we can
  assume $c_i(\theta) > 0$ for $i = 1, \ldots, p$.

  Using the notation
  \begin{equation}\label{eq:Atheta}
    P(\theta)
    = \diag\bigl(\frac{\kappa_1}{c_1(\theta)}, \ldots,
                            \frac{\kappa_p}{c_p(\theta)}\bigr)
      - \Lambda
  \end{equation}
  we can equivalently re-write the condition $\nabla f(\theta) = 0$ as
  \begin{equation}\label{eq:P-sing}
    P(\theta) s(\theta) = 0.
  \end{equation}
  Since
  \begin{equation*}
    P(\theta) = P +
  \diag\bigl(\kappa_1(\frac{1}{c_1(\theta)}-1), \ldots, \kappa_p(\frac{1}{c_p(\theta)}-1)\bigr)
  \end{equation*}
  is the sum of two positive matrices, it is positive and in
  particular non-singular.  Thus, the only solution
  of~\eqref{eq:P-sing} is $s=0$ which implies that the maximum at
  $\theta = \mu$ is the only critical point with $c_i \geq 0$ for
  $i=1, \ldots, p$.  This completes the proof.
\end{proof}

From the proof of proposition~\ref{P:0 is max} we see that, if
$\theta=\mu$ is the global and thus a local maximum of $\phi$, the
matrix $P$ must be positive {\em semi}-definite, \textit{i.e.} the
positivity condition is almost equivalent to $\phi$ having the global
maximum at~$\mu$.  The following corollary gives a sufficient (but not
necessary) condition for the statement to hold; the given condition is
often easier to verify in practice.  Coincidentally, this stronger
condition allows to also identify the minima of the von~Mises
density~$\phi$.

\begin{corollary}\label{C:diag-dom}
  Assume
  \begin{equation}\label{eq:cond}
    \kappa_i > \sum_{j=1}^p |\lambda_{ij}|
        \qquad \mbox{for all $i=1, \ldots, p$.}
  \end{equation}
  Then the global maximum of $\phi = \phi(\quark;\mu,\kappa,\Lambda)$
  is attained at $\theta = \mu$, the global minimum is at $\theta =
  (\mu_1+\pi, \ldots, \mu_p+\pi)$ and these two points are the only
  (local) extrema of $\phi$.
\end{corollary}

\begin{proof}
  By the Gershgorin theorem \citep[Theorem~6.1.1]{HoJo85}, the
  eigenvalues of~$P$ are contained in the union of the closed discs
  $B(\kappa_i, r_i)\subseteq\C$ with radii $r_i = \sum_{j\neq i}
  |-\lambda_{ij}|$ for $i=1, \ldots, p$.  Since $P$ is symmetric, its
  eigenvalues are real and since we have $\kappa_i > \sum_j
  \bigl|\lambda_{ij}\bigr| = r_i$, all eigenvalues of $P$ are
  positive.  Thus the condition of the proposition is satisfied and
  $\theta=0$ is the global maximum of~$\phi$.

  Similarly, the eigenvalues of the matrix $P(\theta)$
  from~\eqref{eq:Atheta} are contained in the union of the closed
  discs $B\bigl(\kappa_i/c_i(\theta), r_i\bigr)\subseteq\R$ with radii $r_i =
  \sum_{j\neq i} |-\lambda_{ij}|$ for $i=1, \ldots, p$.  Since we have
  \begin{equation*}
    \bigl| \frac{\kappa_i}{c_i(\theta)} \bigr|
    \geq \kappa_i
    > \sum_j \bigl|\lambda_{ij}\bigr|
    = r_i,
  \end{equation*}
  none of the discs contain $0$ and the matrix $P(\theta)$ cannot have
  0 as an eigenvalue.  This shows that all solutions
  of~\eqref{eq:P-sing}, \textit{i.e.}\ the critical points of $f$,
  satisfy $s(\theta) = 0$ and thus $c(\theta) \in \{-1, 1\}^p$.

  To classify the critical points, we consider the Hessian matrix $H_f
  = \bigl(\d_{ij}f(\theta)\bigr)_{ij}$.  Invoking the Gershgorin
  theorem again, the eigenvalues of $H_f$ are contained in the union
  of the closed discs with centres $\d_{ii}f(\theta)$ and radii
  $\sum_{j\neq i} |\d_{ij}f(\theta)|$ for $i=1, \ldots, p$.  Using~\eqref{eq:Hess}
  we have
  \begin{equation*}
    \bigl| \d_{ii}f(\theta) \bigr|
    = \bigl| \kappa_i c_i(\theta) \bigr|
    = \bigl| \kappa_i \bigr|
    > \sum_j \bigl| \lambda_{ij} \bigr|
    = \sum_j \bigl| c_i(\theta)\lambda_{ij}c_j(\theta) \bigr|
    = \sum_{j\neq i} |\d_{ij}f(\theta)|,
  \end{equation*}
  none of these discs contain $0$ and thus the circles corresponding
  to $i$ with $c_i = 1$ and with $c_i = -1$ respectively form two disjoint
  groups.  We can conclude that for each $i$ with $c_i = 1$ the matrix
  $H_f$ has a negative eigenvalue and for each $i$ with $c_i = -1$ the
  Hessian has a positive eigenvalue.  Consequently, $\theta=\mu$ is
  the only local maximum of $f$, $\theta=(\mu_1+\pi, \ldots,
  \mu_p+\pi)$ is the local minimum of $f$ and all other critical
  points are saddle points.
\end{proof}

It is easy to see that the statements about the minimum in
corollary~\ref{C:diag-dom} do not necessarily hold under the weaker
assumption from proposition~\ref{P:0 is max}.  For example, the matrix
\begin{equation*}
  \Lambda = \begin{pmatrix}
    0 & -2 & 2 \\ -2 & 0 & 2 \\ 2 & 2 & 0
  \end{pmatrix}
\end{equation*}
has eigenvalues $-4$, $2$ and $2$.  Thus, for $\kappa=(3,3,3)$ the
matrix $P$ is positive (the eigenvalues are $1$, $1$ and $7$), and the
assumption of proposition~\ref{P:0 is max} is satisfied.  On the other
hand, the Hessian matrix of $f$ at $\theta = (\mu_1+\pi, \mu_2+\pi,
\mu_3+\pi)$ is $H_f(\theta) = \diag(\kappa_1, \kappa_2, \kappa_3) +
\Lambda$ and, since this matrix is not positive semi-definite (the
eigenvalues are $-1$, $5$ and $5$), the minimum of the distribution
cannot be at $(\mu_1+\pi, \mu_2+\pi, \mu_3+\pi)$.


\section{Low Concentration}
\label{S:low}

In this section we consider the case of ``small''~$\kappa$.  In this
case the structure of the extrema of a~$\MVM(\mu,\kappa,\Lambda)$
distribution is much more complicated than for the concentrated case.
We illustrate some of the possible scenarios with the help of
examples, starting with the boundary case $\kappa = (0, 0,
\ldots, 0)$ and then considering small but non-zero~$\kappa$.

The following lemma shows that for $\kappa = 0$ the case of a single
global maximum can never occur.

\begin{lemma}\label{L:cube1}
  For $\kappa=0$, the following statements hold:
  \begin{enumerate}
  \item The density of the multivariate von Mises distribution
    $\MVM(\mu,0,\Lambda)$ takes its maximal value on the set
    $\{\frac12\pi, \frac32\pi \}^p \subseteq \T^p$, \textit{i.e.}
    \begin{equation*}
      \sup_{\theta\in\T^p} \phi(\theta;\mu,0,\Lambda)
      = \sup_{\theta\in \{\frac12\pi, \frac32\pi \}^p} \phi(\theta;\mu,0,\Lambda).
    \end{equation*}
  \item If $\theta$ is a maximum, then so is $(\theta_1+\pi, \ldots,
    \theta_p+\pi)$.  In particular the number of isolated maxima of
    $f$ is always even (and thus cannot be~$1$).
  \end{enumerate}
\end{lemma}

\begin{proof}
  Without loss of generality, we can assume $\mu=0$.  Let
  $\theta^*\in\T^p$ be a global maximum of $\phi(\quark;0,0,\Lambda)$.
  As in proposition~\ref{P:0 is max}, this is equivalent to $\theta^*$
  being a maximum of the function~$f$ from equation~\eqref{eq:f}.
  Since we assume $\kappa=0$, the formula for~$f$ simplifies to
  \begin{equation}\label{eq:f-no-kappa}
    f(\theta) = \frac12 s(\theta)^\top\Lambda s(\theta)
  \end{equation}
  and the partial derivatives of~$f$ are given by
  \begin{equation}\label{eq:fprime-no-kappa}
    \d_if(\theta) = c_i(\theta) \sum_{k=1}^p \lambda_{ik} s_k(\theta).
  \end{equation}

  Let $i\in\{1, 2, \ldots, p\}$.  Since $\lambda_{ii} = 0$, the value
  $\sum_{k=1}^p \lambda_{ik} s_k(\theta)$ does not depend on
  $\theta_i$ and thus $\theta_i \mapsto \d_if(\theta)$ can only change
  sign at the points $\theta_i = \frac12\pi, \frac32\pi$.
  Consequently, $\theta_i \mapsto f(\theta)$ changes monotonically
  between the values $\theta_i = \frac12\pi, \frac32\pi$.  Defining
  $\theta^+$ and $\theta^-$ by $\theta^+_i = \frac12\pi$, $\theta^-_i
  = \frac32\pi$, and $\theta^+_j = \theta^-_j = \theta^*_j$ for $j\neq
  i$ this shows that one of the two inequalities $f(\theta^+) \geq
  f(\theta^*) \geq f(\theta^-)$ and $f(\theta^-) \geq f(\theta^*) \geq
  f(\theta^+)$ holds.  Since $\theta^*$ is a global maximum of $f$,
  equality holds in the upper bound and thus either $\theta^+$ or
  $\theta^-$ is also a global maximum.  By repeating this procedure
  for $i=1,2, \ldots, p$ we find a global maximum where each
  coordinate is in the set~$\{ \frac12\pi, \frac32\pi \}$.  This
  completes the proof of the first statement.

  The second statement is a direct consequence of the fact that the
  function $f$ from \eqref{eq:f-no-kappa} is invariant under the map
  $\theta \mapsto \theta + (\pi, \ldots, \pi)$.
\end{proof}

\begin{lemma}\label{L:cube}
  Let $\kappa=0$ and $\Lambda \neq 0$.  Then every global maximum $\theta$ of
  $\phi(\theta;\mu,0,\Lambda)$ satisfies $\|s(\theta)\|_\infty = 1$.
\end{lemma}

\begin{proof}
  Since the trace of a matrix equals the sum of its eigenvalues and
  since $\Lambda$ is a non-zero matrix with zero trace, $\Lambda$ must
  have a strictly positive eigenvalue $\lambda$.  Let $x$ be a
  corresponding eigenvalue with $\|x\|_\infty \leq 1$.  Then we can
  find $\theta = (\theta_1, \ldots, \theta_p)$ with $\sin(\theta_i) =
  x_i$ for $i=1,2, \ldots, p$.  This vector $\theta$ satisfies
  \begin{equation*}
    f(\theta) = \frac12 s(\theta)^\top\Lambda s(\theta)
    = \frac{\lambda}{2} s(\theta)^\top s(\theta)
    > 0.
  \end{equation*}
  Consequently the maximal value of~$f$ is strictly positive.

  Now let $\theta\in\T^p$ with $f(\theta) > 0$ and
  $\|s(\theta)\|_\infty < 1$, \textit{i.e.}\
  $\bigl|s_i(\theta)\bigr|<1$ for all $i\in \{1, \ldots, p\}$.  Let $c
  = 1 / \|s(\theta)\|_\infty > 1$ and $\tilde s = c s(\theta)$.  Since
  $\| \tilde s \|_\infty = 1$, we can find $\tilde\theta \in \T^p$
  with $\sin(\tilde\theta_i) = \tilde s_i$ for $i=1,2, \ldots, p$.
  This point satisfies
  \begin{align*}
    f(\tilde\theta)
    = \frac12 \tilde s^\top\Lambda \tilde s
    = c^2 \frac12 s(\theta)^\top\Lambda s(\theta)
    > f(\theta).
  \end{align*}
  Therefore, $\theta$ cannot have been a maximum of~$f$.
\end{proof}

\medskip

\noindent
\textbf{Example~1} ($\kappa = 0$, two isolated modes).
Consider a $\MVM(\mu, 0, \Lambda)$ distribution
with
\begin{equation*}
  \Lambda
  = \begin{pmatrix}
     0 & 1.75 & 0.77 \\
    1.75 & 0 & 0.06 \\
    0.77 & 0.06 & 0
  \end{pmatrix}.
\end{equation*}
Since $\kappa = 0$,
lemma~\ref{L:cube} applies and shows that all maxima of
$\phi(\quark;\mu,0,\Lambda)$ correspond to $\theta$ where
$s(\theta)$ lies on the surface of the cube $Q = [-1,1]^3$.
Thus, we can find
the local extrema of $f$ by first finding the local extrema of $g(s) =
\frac12 s^\top\Lambda s$ on the surface of~$Q$ and then
identifying the corresponding values~$\theta$.  To aid with finding
the maxima of~$g$, figure~\ref{fig:random-cube} shows a plot
of $g$ on the (unwrapped) surface of~$Q$.  In the figure,
    the top-most square
    corresponds to $s_3 = +1$, the centre square to $s_2 = -1$, the
    right-most square to $s_1=+1$ and so on.  One can see that the
    distribution has two modes, corresponding to $s(\theta) =
    (-1,-1,-1)$ and $s(\theta) = (+1,+1,+1)$.

\begin{figure}
  \begin{center}
    \includegraphics{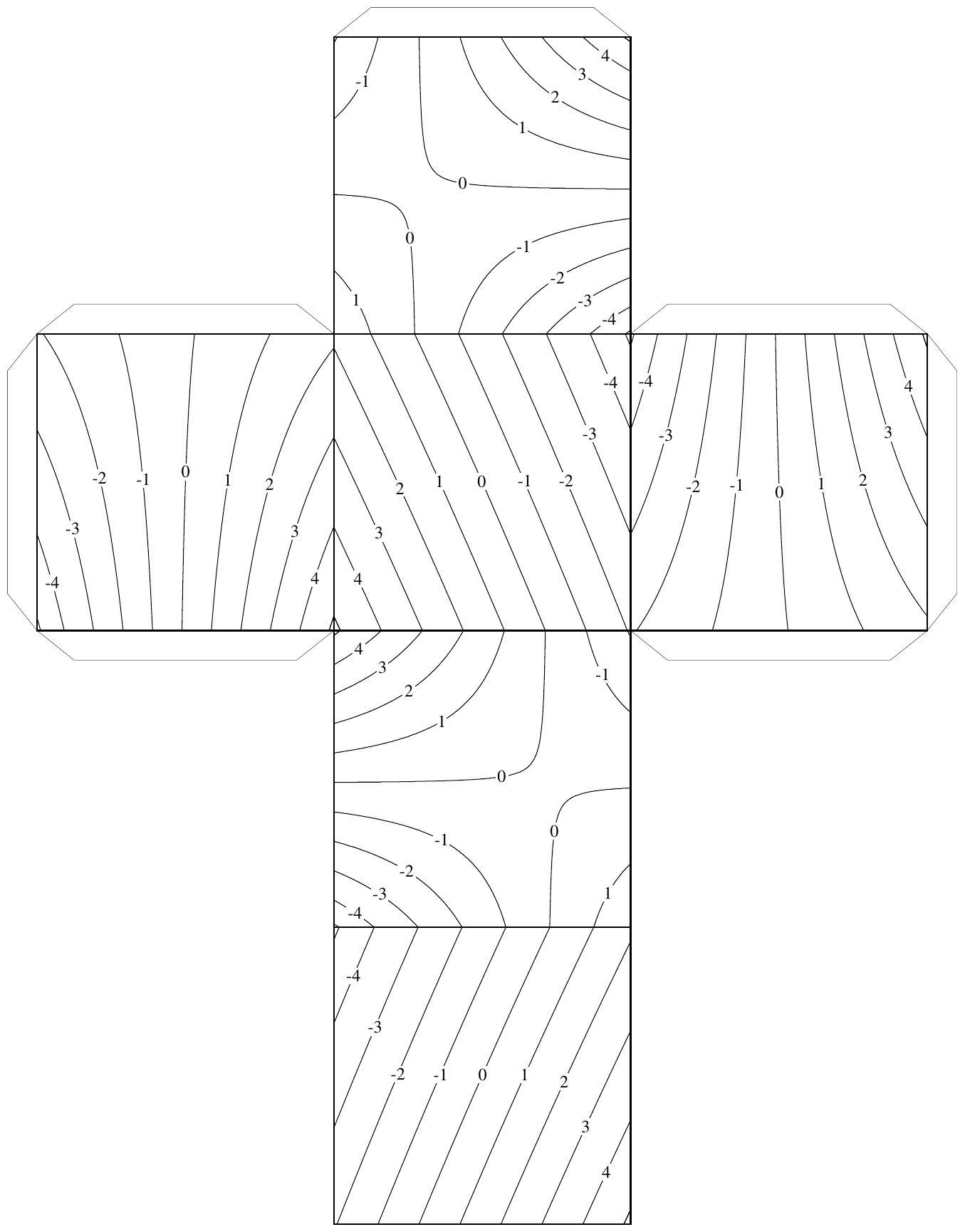}
  \end{center}
  \caption{\label{fig:random-cube}\it
    Visualisation of a von~Mises density from example~1,
    as a function of $s(\theta)$ restricted to the surface of
    the cube~$[0,1]^3$.  The plot shows that the distribution has
    two modes.}
\end{figure}

\medskip

\noindent
\textbf{Example~2} ($\kappa = 0$, one extended mode).
Consider a $\MVM(\mu, 0, \Lambda)$ distribution
with
\begin{equation*}
  \Lambda
  = \begin{pmatrix}
    0 & -1 & 1 \\
    -1 & 0 & 1 \\
    1 & 1 & 0
  \end{pmatrix}.
\end{equation*}
We can find the modes of this distribution in the same way as we did
in example~1, the corresponding plot of~$g$ is shown in
figure~\ref{fig:special-cube}.  The figure shows that the density of
this $\MVM(\mu, 0, \Lambda)$ distribution has an extended maximum
which forms a loop on the surface of the cube.  Figure~\ref{fig:special-cube}
shows the density as a function of~$s(\theta)$.  To give an idea of
the distribution of the corresponding angles~$\theta_1, \theta_2, \theta_3$
themselves, we show
a scatter plot of a sample in figure~\ref{fig:special-scatter}.  While this
(much more conventional) diagram shows the distribution of the sample
clearly, comparison with figure~\ref{fig:special-cube} makes it
clear that the structure of the mode is difficult to
understand from a scatter plot alone.

\begin{figure}
  \begin{center}
    \includegraphics{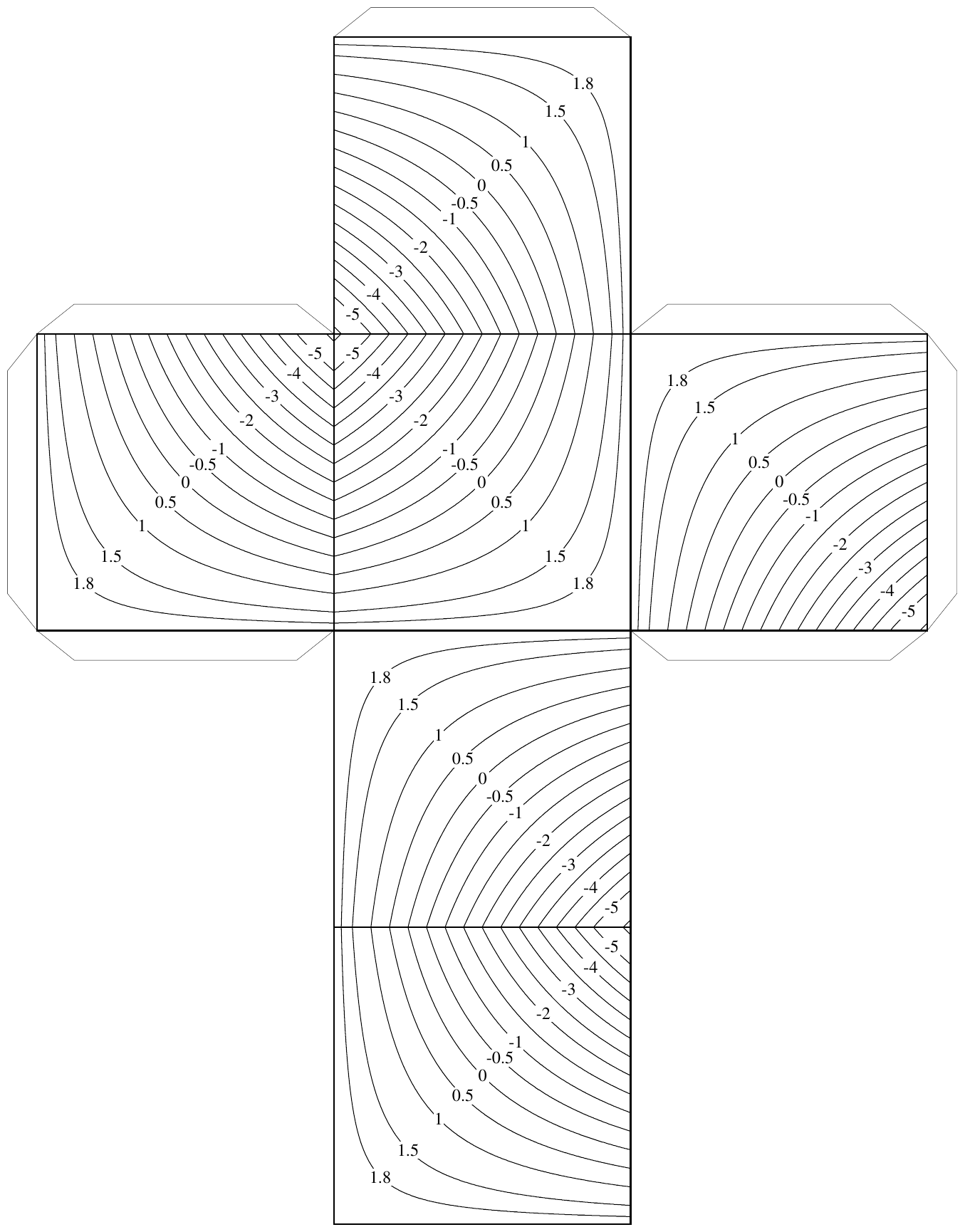}
  \end{center}
  \caption{\label{fig:special-cube}\it Visualisation of a von~Mises
    density from example~2.  One can see that the distribution has an
    extended maximum which loops around the cube in a ``zig-zag
    belt''.}
\end{figure}

\begin{figure}
  \begin{center}
    \includegraphics{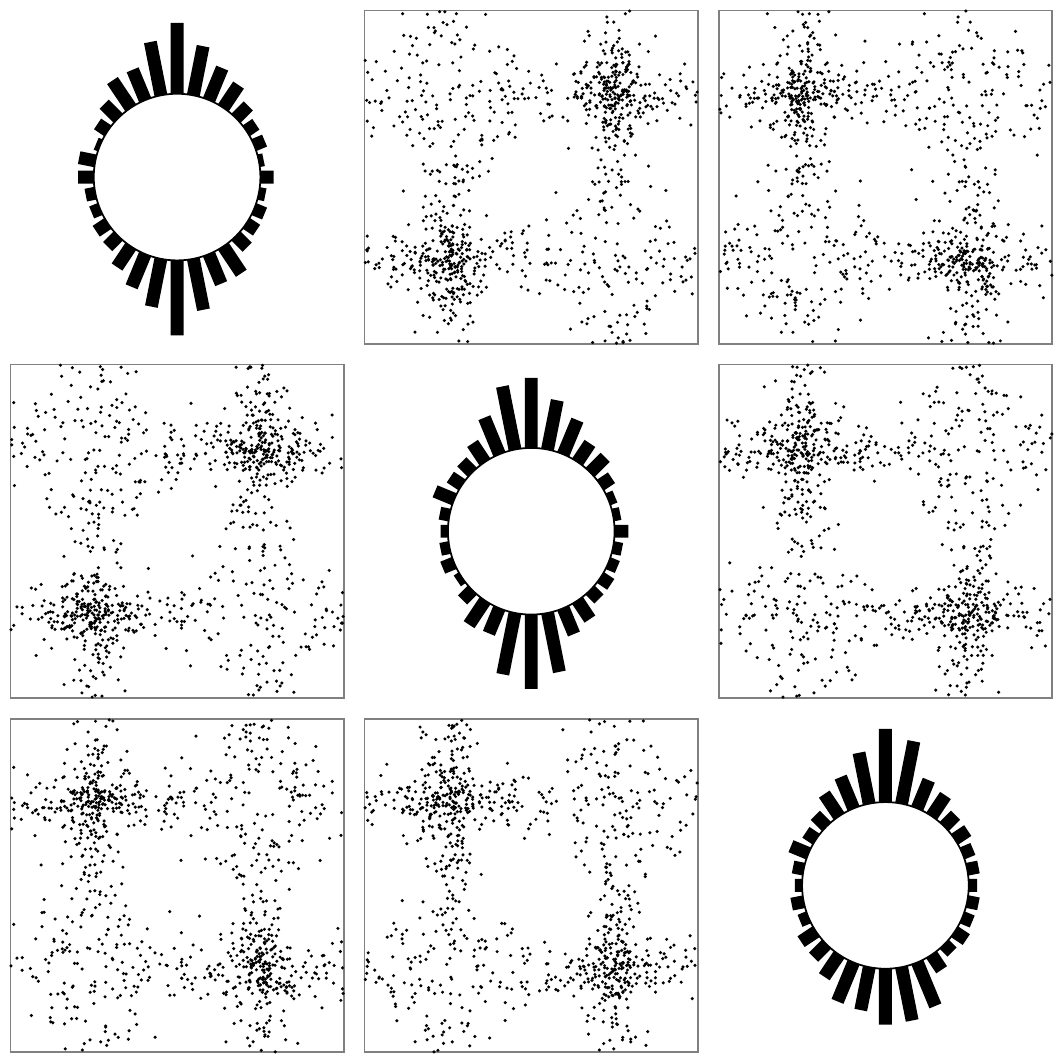}
  \end{center}
  \caption{\label{fig:special-scatter}\it
    Scatter plot of 1000 samples from the distribution from
    example~2.  To make the structure of the
    maximum more visible, the matrix $\Lambda$ was multiplied by 10,
    \textit{i.e.}\ the plotted sample is from a $\MVM(0, 0, 10\,\Lambda)$
    distribution.  The
    regions where the scatter plots have higher
    intensity are not isolated modes of the distribution but are
    artefacts caused by the projection of $\T^3$ onto $\T^2$
    where straight segments of the extended maximum are seen ``head-on''.}
\end{figure}

\bigskip

The case of small, non-zero $\kappa$ can be seen as a perturbation of
the case $\kappa=0$.  Such a perturbation would normally just shift
the extrema of the density around, but the following example shows
that such a perturbation can also break a spatially extended maximum
into a set of isolated maxima, thus increasing the number of modes.

\medskip

\noindent
\textbf{Example~3} ($\kappa>0$, six isolated modes).
The maximum of the von~Mises distribution illustrated in
figure~\ref{fig:special-cube} lives on a ``ring'' formed as the union
of six lines in $\T^3$, aligned with the grid $\{\frac{1}{2} \pi,
\frac{3}{2} \pi\}^3$.  Since the $c_i$ are zero on the grid and take
their maxima between the grid points, we would expect that adding a
perturbation term $\kappa^\top c(\theta)$ with small $\kappa$ will not
only shift these lines, but will also collapse this extended maximum
into a collection of isolated maxima which live on the shifted lines,
at the point where the perturbation was maximal.  The following,
explicit example gives a von~Mises distribution in $\T^3$ with six
isolated maxima.

Let $\eta>0$ and $\eps = \sin(\eta)$.  Define
\begin{equation*}
  \kappa = \begin{pmatrix}
    \eps \\ \eps \\ \eps
  \end{pmatrix},
  \qquad
  \Lambda = \begin{pmatrix}
    0 & -1 & 1 \\
    -1 & 0 & 1 \\
    1 & 1 & 0
  \end{pmatrix}.
\end{equation*}
We will show that, for small enough $\eta$, the function $f(\theta) =
\kappa^\top c(\theta) + \frac12 s(\theta)^\top\Lambda s(\theta)$ has
local maxima at the six points $\theta^1, \ldots, \theta^6\in\T^3$
given by the following table.
\begin{equation*}
  \begin{array}{c|ccc|ccc|ccc}
\ell &
  \theta^\ell_1 & \theta^\ell_2 & \theta^\ell_3 &
  s_1(\theta^\ell) & s_2(\theta^\ell) & s_3(\theta^\ell) &
  c_1(\theta^\ell) & c_2(\theta^\ell) & c_3(\theta^\ell) \\[1.5ex]
1 &
  0 & \frac{3}{2} \pi + \eta & \frac{3}{2} \pi + \eta &
  0 & - \sqrt{1 - \eps^{2}} & - \sqrt{1 - \eps^{2}} &
  1 & \eps & \eps \\
2 &
  \frac{1}{2} \pi - \eta & \frac{3}{2} \pi + \eta & 0 &
   \sqrt{1 - \eps^{2}} & - \sqrt{1 - \eps^{2}} & 0 &
   \eps & \eps & 1 \\
3 &
  \frac{1}{2} \pi - \eta & 0 & \frac{1}{2} \pi - \eta &
  \sqrt{1 - \eps^{2}} & 0 & \sqrt{1 - \eps^{2}} &
  \eps & 1 & \eps \\
4 &
  0 & \frac{1}{2} \pi - \eta & \frac{1}{2} \pi - \eta &
  0 & \sqrt{1 - \eps^{2}} & \sqrt{1 - \eps^{2}} &
  1 & \eps & \eps \\
5 &
  \frac{3}{2} \pi + \eta & \frac{1}{2} \pi - \eta &
  0 & - \sqrt{1 - \eps^{2}} & \sqrt{1 - \eps^{2}} & 0 &
  \eps & \eps & 1 \\
6 &
  \frac{3}{2} \pi + \eta & 0 & \frac{3}{2} \pi + \eta &
  - \sqrt{1 - \eps^{2}} & 0 & - \sqrt{1 - \eps^{2}} &
  \eps & 1 & \eps
  \end{array}
\end{equation*}
For the convenience of the reader, the table also gives the vectors
$s(\theta^\ell)$ and $c(\theta^\ell)$ for $\ell = 1, \ldots, 6$.  By
substituting these values into the formula for $\d_i f$ from
equation~\eqref{eq:grad}, it is easy to check that $\nabla
f(\theta^\ell) = 0$ for $\ell = 1, \ldots, 6$ and thus all six points
are critical points of~$f$.

Substituting the values from the table into the formulas for
$\d_{ij}f$ from \eqref{eq:Hess}, we can compute the value of the
Hessian matrix $H_\ell = H_f(\theta^\ell)$ for $\ell = 1, \ldots, 6$.
The results are as follows:
\begin{small}
\begin{equation*}
H_1 = H_4 = \begin{pmatrix}
  - \eps & - \eps & \eps \\ - \eps & -1 & \eps^{2} \\ \eps & \eps^{2} & -1
\end{pmatrix}, \quad
H_2 = H_5 = \begin{pmatrix}
  -1 & - \eps^{2} & \eps \\ - \eps^{2} & -1 & \eps \\ \eps & \eps & - \eps
\end{pmatrix}, \quad
H_3 = H_6 = \begin{pmatrix}
  -1 & - \eps & \eps^{2} \\ - \eps & - \eps & \eps \\ \eps^{2} & \eps & -1
\end{pmatrix}.
\end{equation*}
\end{small}%
It can be checked that each of these matrices has eigenvalues
$\lambda_1 = -1 + \CO(\eps^2)$, $\lambda_2 = -1 + \CO(\eps^2)$, and
$\lambda_3 = -\eps + \CO(\eps^2)$.  Thus, for small enough $\eps>0$,
all six points are local maxima as required.

\section{Sampling}
\label{S:sampling}

In this section we discuss a simple method to generate samples from a
$\MVM(\mu,\kappa,\Lambda)$ distribution, using the rejection sampling
algorithm \citep[Corollary~2.17]{RoCa04a}.  The method is restricted
to small or moderate $p$, but works well for the case of high concentration.
We assume that the matrix
\begin{equation*}
  P = \diag(\kappa_1, \ldots, \kappa_p) - \Lambda
\end{equation*}
is positive definite.

Without loss of generality we
can assume $\mu=0$, the general case is then obtained by a simple
shift.   We denote the smallest eigenvalue of $P$ by
$\lmin > 0$.
The proposed algorithm uses independent angles $\theta_1,
\theta_2, \ldots, \theta_p$ as proposals, distributed with
density
\begin{equation*}
  g(\theta)
  = \prod_{i=1}^p \frac{\exp\bigl(\frac{\lmin}{4}\cos(2\theta)\bigr)}{2\pi I_0\bigl(\frac{\lmin}{4}\bigr)}.
\end{equation*}
This is the independent product of one-dimensional von Mises distributions,
modified by replacing the angle $\theta$ by~$2\theta$.
Since we can efficiently generate samples $\tilde\theta_i$ from a
one-dimensional von~Mises distribution $\VM(0, \lmin/4)$
\citep[\textit{e.g.}][]{BeFi79}, we can obtain samples from the
density $g$ by taking $\theta_i = \tilde\theta / 2$ with probability
$1/2$ and $\theta_i = \tilde\theta / 2 + \pi$ else.

The target density is the density of the multivariate von Mises
distribution $\MVM(0,\kappa,\Lambda)$, \textit{i.e.}\ it is
proportional to
\begin{equation*}
  f(\theta)
  = \exp\Bigl(
    \kappa^\top c(\theta)
    + \frac12 s(\theta)^\top \Lambda s(\theta)
  \Bigr).
\end{equation*}
Using the inequalities $\cos(\theta) + \sin(\theta)^2 / 2 \leq 1$ and
$s(\theta)^\top P s(\theta) \geq \lmin s(\theta)^\top s(\theta)$,
we find
\begin{align*}
  f(\theta)
  &= \exp\Bigl(
    \kappa^\top c(\theta)
    + \frac12 s(\theta)^\top \Lambda s(\theta)
  \Bigr) \\
  &= \exp\Bigl(
    \sum_{i=1}^p \kappa_i \bigl(c_i(\theta) + \frac12 s_i(\theta)^2\bigr)
    - \frac12 s(\theta)^\top P s(\theta) \Bigr) \\
  &\leq \exp\Bigl(
    \sum_{i=1}^p \kappa_i
    - \frac{\lmin}{2} s(\theta)^\top s(\theta)
  \Bigr).
\end{align*}
Finally, since $\cos(2x) = 1 - 2 \sin(x)^2$, we can rewrite this
expression as
\begin{align*}
  f(\theta)
  &\leq \exp\Bigl( - \frac{p \lmin}{4} + \sum_{i=1}^p \kappa_i \Bigr)
     \cdot \exp\bigl( \frac{\lmin}{4} \sum_{i=1}^p c_i(\theta) \bigr) \\
  &= \exp\Bigl( - \frac{p \lmin}{4} + \sum_{i=1}^p \kappa_i \Bigr)
     \cdot \bigl( 2\pi I_0(\frac{\lmin}{4})\bigr)^p
     \cdot g(\theta) \\
  &=: C g(\theta).
\end{align*}
Thus we have found a constant $C$ with $f \leq C g$ and the rejection
sampling algorithm can be applied.

In the rejection sampling algorithm, a proposal $\theta$ is accepted
with probability $f(\theta) / Cg(\theta)$, \textit{i.e.}\ with
probability
\begin{align*}
  p(\theta)
  &= \frac{\exp\Bigl(
    \kappa^\top c(\theta)
    + \frac12 s(\theta)^\top \Lambda s(\theta)
  \Bigr)}
         {\exp\Bigl(
    \sum_{i=1}^p \kappa_i
    - \frac{\lmin}{2} s(\theta)^\top s(\theta)
  \Bigr)} \\
  &= \exp\Bigl( \sum_{i=1}^p \kappa_i\bigl(c_i - 1\bigr)
    + \frac12 s^\top (\Lambda + \lmin I) s \Bigr)
\end{align*}
where $I$ is the $p\times p$ identity matrix.  Thus, the following
algorithm can be used to generate samples of a
$\MVM(\mu,\kappa,\Lambda)$ distribution when $P$ is positive:

\begin{enumerate}
\item\label{alg:start} Generate random variables
  \begin{align*}
    &\tilde\theta_1, \ldots, \tilde\theta_p \sim \VM(0, \lmin/4) \\
    &\delta_1, \ldots, \delta_n
      \quad\mbox{with $P(\delta_i = 0) = P(\delta_i = \pi) = 1/2$} \\
    &U \sim \CU\bigl([0,1]\bigr),
  \end{align*}
  all independent of each other.
\item Let $s_i = \sin(\theta_i)$ and $c_i = \cos(\theta_i)$ for
  $i=1,2,\ldots, p$.
\item If the condition
  \begin{equation*}
    U \leq \exp\Bigl( \sum_{i=1}^p \kappa_i\bigl(c_i(\theta) - 1\bigr)
    + \frac12 s(\theta)^\top (\Lambda + \lmin I) s(\theta) \Bigr)
  \end{equation*}
  is satisfied, output $\theta = (\theta_1+\mu_1, \theta_2+\mu_2, \ldots,
  \theta_p+\mu_p)$ (\textit{i.e.}\ the proposal is accepted).
\item Return to step~\ref{alg:start}.
\end{enumerate}

We note that the algorithm still works when the eigenvalue $\lmin$ is
replaced by a lower bound $0 < \hat\lambda_{\rm min} \leq \lmin$ for
the eigenvalues of~$P$.  This allows to apply the algorithm in
situations where the eigenvalues of $P$ are not exactly known.

The efficiency of this algorithm is determined by its acceptance
rate: If $Z$ is the normalisation constant which makes $\frac{1}{Z} f$
a probability density, then each proposal is accepted with probability
$Z/C$.  From~\citet[equation~(3)]{MaKeZhaTayHa} we know that, for high
concentration, we have
\begin{equation*}
  Z \approx (2\pi)^{p/2} |P|^{-1/2} \exp\Bigl(\sum_{i=1}^p \kappa_i\Bigr)
\end{equation*}
where $|P|$ is the determinant of the matrix~$P$.  From
\citet[formula~9.7.1]{AbraSte64} we know
\begin{equation*}
  \sqrt{2\pi\kappa} \, \e^{-\kappa} I_0(\kappa) \longrightarrow 1
\end{equation*}
as $\kappa\to\infty$.  Consequently, the asymptotic acceptance
probability for high concentration is
\begin{equation}\label{eq:acceptance-prob}
  \begin{split}
  \frac{Z}{C}
  &\approx \frac{(2\pi)^{p/2} |P|^{-1/2} \exp\bigl(\sum_{i=1}^p \kappa_i\bigr)}%
     {\exp\bigl( - p \lmin/4 + \sum_{i=1}^p \kappa_i \bigr)
     \cdot (2\pi)^{p/2} (4/\lmin)^{p/2} \exp\bigl( p\lmin/ 4\bigr)} \\
  &= \frac{1}{2^p} \cdot \sqrt{\frac{\lmin^p}{|P|}}.
  \end{split}
\end{equation}
The proposed algorithm will be efficient if this probability is not to
small.  Considering the first factor on the right-hand side
of~\eqref{eq:acceptance-prob}, we see that the method only can be
expected to perform well for sufficiently small values of~$p$.  The factor
$1/2^p$ is expected, since the proposal distribution has $2^p$ modes,
whereas the target distribution has only one.  Since the determinant
$|P|$ equals the product of all $p$ eigenvalues of~$p$ (the smallest
of which is $\lmin$), the second factor on the right-hand side
of~\eqref{eq:acceptance-prob} is big, if the eigenvalues of $P$ are
all of the same magnitude, \textit{i.e.}\ if the mode of the
distribution is approximately rotationally symmetric.

\bigskip

\textbf{Acknowledgements.}  The authors wish to thank John Kent for
many helpful discussions.

\bibliographystyle{abbrvnat}
\bibliography{modes}

\end{document}